\documentclass[12pt]{amsart}
\usepackage[english]{babel}
\usepackage{amssymb}
\usepackage{latexsym,amsfonts,amsmath,enumerate,graphics,enumerate,amsthm,curves}
\usepackage[affil-it]{authblk}

\let\epsilon=\varepsilon
\newtheorem{theorem}{Theorem}[section]
\newtheorem{lemma}[theorem]{Lemma}
\newtheorem{corollary}[theorem]{Corollary}

\newtheorem{proposition}[theorem]{Proposition}

\theoremstyle{definition}
\newtheorem{definition}[theorem]{Definition}
\newtheorem{example}[theorem]{Example}
\newtheorem{remark}[theorem]{Remark}

\newcommand{\kommentaar}[1]{\begin{item}\item[]{\small\textsf{#1}}\end{itemize}}
\numberwithin{equation}{section}

\let\epsilon=\varepsilon
\let\phi=\varphi
\let\theta=\vartheta

\newcommand{\comment}[1]{}

\def\sqr#1#2{{\,\vcenter{\vbox{\hrule height.#2pt\hbox{\vrule width.#2pt
height#1pt \kern#1pt\vrule width.#2pt}\hrule height.#2pt}}\,}}
\def\bo{\sqr44\,}

\makeatletter
\def\@maketitle{%
  \newpage
  \null
  \vskip 2em%
  \begin{center}%
  \let \footnote \thanks
    {\Large\bfseries \@title \par}%
    \vskip 1.5em%
    {\normalsize
      \lineskip .5em%
      \begin{tabular}[t]{c}%
        \@author
      \end{tabular}\par}%
    \vskip 1em%
    {\normalsize \@date}%
  \end{center}%
  \par
  \vskip 1.5em}
\makeatother

\begin{document}

\title[]{\large\bf  Infinite dimensional Jordan algebras and  symmetric cones}

\author{Cho-Ho Chu}

\address{School of Mathematical Sciences, Queen Mary, University of  London, London E1 4NS, UK}
 \email{ c.chu@qmul.ac.uk}

\date{}

\maketitle

Abstract.  A celebrated result of Koecher and
Vinberg asserts the one-one correspondence between the finite dimensional formally real
Jordan algebras and Euclidean symmetric cones. We extend this result to the infinite dimensional setting.\\



\section{Introduction}\label{sec1}

Finite dimensional formally real Jordan algebras were first introduced by Jordan, von Neumann and Wigner
 for quantum mechanics formalism in \cite{jvw}, where these algebras were completely classified. Since then, many far reaching connections to Lie algebras, geometry and analysis have been found. One such connection to geometry is the seminal result of
 Koecher \cite{k} and Vinberg \cite{v}, which establishes the one-one correspondence between the formally real Jordan algebras
 and a class of Reimmanian symmetric spaces, namely, the symmetric cones. The latter plays a useful role in the study of
 automorphic functions on bounded homogeneous domains in complex spaces  and harmonic analysis
(see, for example, \cite{fk,gpv,v} and references therein).

In recent decades, infinite dimensional Jordan algebras and Jordan triple systems have gradually become a significant part of the theory of bounded symmetric domains. While much of the theory of finite dimensional bounded
symmetric domains, which are Riemannian symmetric spaces, can be extended to infinite dimension via Jordan theory, an infinite dimensional generalisation
of symmetric cones and the result of Koecher and Vinberg has not yet been accomplished.
 Our objective in this paper is to carry out this task  by introducing infinite dimensional
symmetric cones  and show in Theorem \ref{symm} that they correspond exactly to a class of infinite dimensional real Jordan algebras with identity, called
unital {\it JH-algebras}. In finite dimensions, the unital JH-algebras are exactly the formally real Jordan algebras and
our result is identical to that of Koecher and Vinberg.

Although our approach is a natural extension of the finite dimensional one, there are some infinite dimensional
pitfalls in Lie theory and other obstructions that different arguments are required to circumvent,
for instance, a closed subgroup of an infinite
dimensional Lie group need not be a Lie group in the relative topology \cite{hk}
and an infinite dimensional orthogonal group
need not be compact and lacks an invariant measure. Unlike the finite dimensional case,
the order-unit structures play a prominent role in
infinite dimension and we make use of the weak topology as well as Kakutani's fixed-point theorem to achieve the
final result. This approach also provides some new perspectives for the finite dimensional case.
Our focus, however, is to present a  proof of
Theorem \ref{symm} as simply as possible, but not discuss all its ramifications.

For completeness, we review briefly some relevant basics of Jordan algebras and refer
to \cite{chu,up} for more details.
In what follows, a Jordan algebra $\mathcal{A}$ is a real vector
space, which can be infinite dimensional, equipped with a bilinear product $(a,b) \in \mathcal{A} \times \mathcal{A} \mapsto ab \in \mathcal{A}$
that is commutative and satisfies the {\it Jordan identity}
$$a(ba^2) = (ab)a^2 \qquad (a,b \in \mathcal{A}).$$
A Jordan algebra is called {\it unital} if it contains an identity. There are two fundamental linear operators on a
Jordan algebra $\mathcal{A}$, namely, the left multiplication $L_a: x\in \mathcal{A} \mapsto ax \in \mathcal{A}$ and the quadratic map
$Q_a : x\in \mathcal{A} \mapsto \{a,x,a\} \in \mathcal{A}$, where the Jordan triple product $\{\cdot, \cdot,\cdot\}$ is defined by
$$\{a,b,c\} = (ab)c + a(bc) - b(ac) \qquad (a,b,c \in \mathcal{A}).$$
We have the identities
$$Q_a = 2L_a^2 - L_{a^2}, \quad Q_a^2 = Q_{a^2} \qquad (a\in \mathcal{A}).$$
The operator $L_{ab} + [L_a, L_b] : \mathcal{A} \rightarrow \mathcal{A}$, where $[\cdot,\cdot]$ denotes the Lie brackets,
 is often denoted by $a\bo b$ and is called a {\it box operator}.

An element $a$ in a Jordan algebra with identity $e$ is called {\it invertible} if there exists an
element  $a^{-1}\in \mathcal{A}$ (which is necessarily unique)  such that
$aa^{-1}=e$ and $(a^2)a^{-1}=a$. This is equivalent to the invertibility of the quadratic operator $Q_a$, in which case
  $a^{-1}=Q_a^{-1}(a)$.
If the left multiplication $L_a$ is invertible, then $a$ is invertible with inverse $a^{-1}= L_a^{-1}(e)$.

A Jordan algebra $\mathcal{A}$ is called {\it formally real} if $a_1^2 + \cdots +a_n^2 =0$ implies
$a_1 = \cdots =a_n =0$ for any $a_1, \ldots, a_n \in \mathcal{A}$. A finite dimensional formally real Jordan algebra $\mathcal{A}$ is
necessarily unital (cf.\,\cite[Proposition 1.1.13]{chu}) and contains an abundance of {\it idempotents}, which are
  elements $p$ satisfying $p^2 =p$ (cf.\,\cite[Theorem 1.1.14]{chu}). It is a real Hilbert space in the trace norm
$\|a\|^2 = {\rm trace}\, (a \bo a)$ for $a\in \mathcal{A}$.

Following \cite{nomura}, we call a real Jordan algebra $\mathcal{H}$ a {\it JH-algebra} if it is also a Hilbert space
in which the inner product, always denoted by $\langle \cdot,\cdot\rangle$, is {\it associative}, that is,
\[ \langle ab,c\rangle = \langle b,ac\rangle \qquad (a,b,c \in \mathcal{H}).\]
A finite dimensional JH-algebra is called {\it Euclidean} in \cite{fk}.
In fact, the finite dimensional formally
real Jordan algebras are exactly the Euclidean Jordan algebras with identity \cite[Lemma 2.3.7]{chu}.
JH-algebras are examples of non-associative H*-algebras which have been studied by many authors and
references are detailed in \cite[p.\,222]{rod}.

Throughout, all vector spaces are over the real scalar field unless stated otherwise.

\section{Symmetric cones}

Let $V$ be a real vector space.
By a {\it cone} $C$ in $V$, we mean a  nonempty subset of $V$ satisfying
(i) $C + C \subset C$ and (ii) $\alpha C \subset C$ for all $\alpha >0$. We note that
a cone is necessarily convex. A cone $C$ is called
{\it proper} if $C\cap -C = \{0\}$.
 The partial ordering on $V$ induced by a proper cone $C$ will be denoted by
$\leq_C$, or by $\leq$ if $C$ is understood, so that $x \leq y$ whenever $y-x \in C$. Conversely, if $V$ is equipped with a partial
ordering $\leq$, we let $V_+ = \{v\in V: 0 \leq v\}$ denote the corresponding proper cone.

Given a real topological vector space $V$ and a set $E\subset V$, we will denote its closure and interior by $\overline E$ and ${\rm int}\,E$ respectively.
If $C$ is a cone in $V$, then its closure $\overline C$ is also a cone. If $C$ is an open cone,
then we have ${\rm int}\,\overline C = C$.
For completeness, we include a proof of this fact in the next lemma.

\begin{lemma}\label{2.1} Let $C$ be an open  convex set in a real topological vector space $V$. Then ${\rm int}\,\overline C = C$.
\end{lemma}
\begin{proof} There is nothing to prove if $C$ is empty. Pick any $q\in C$.
Let $p \in {\rm int}\,\overline C$. Then $p$ is an internal point of $\overline C$, that is, every line
through $p$ meets $\overline C$ in a set containing an interval around $p$ (cf.\,\cite[p.410,\,413]{d}).
In particular, for the line joining $p$ and $q$, there exists $\delta \in (0,1)$ such that $p\pm \delta (q-p) \in \overline C$.
Since $C$ is open and $q$ is an interior point of $C$, we have $\lambda (p-\delta (q-p)) + (1-\lambda)q \in C$ for
$0<\lambda <1$ (cf.\,\cite[p.413]{d}). Hence $$p = \frac{1}{1+\delta}(p- \delta(q-p))+ \frac{\delta}{1+\delta} q
\in C.$$
\end{proof}

Given an open cone $C$ such that $\overline C \cap -\overline C =\{0\}$, we must have $0\notin C$
since every point in C is an internal point of $\overline C$.

Let $V$ be a real vector space equipped with a proper cone $V_+$ and induced partial ordering
 $\leq$. An element $e\in V$ is called an {\it order-unit} if
 $$V= \bigcup _{\lambda >0} \{x\in V: -\lambda e \leq x \leq \lambda e\}.$$
We call $V$ an {\it order-unit space} if it admits an order-unit. The real field $\mathbb{R}$ with the usual partial
ordering is an order-unit space with order-unit $1$. We note that an order-unit space $V$, with order-unit $e$,
is determined by its cone
$V_+$ completely in that $V= V_+ - V_+$. Indeed, each $x\in V$ with $-\lambda e \leq x \leq \lambda e$ can be written
as $x = x_1 - x_2$, where $2x_1 = \lambda e + x \in V_+$ and $ 2x_2 = \lambda e -x \in V_+$.

An order-unit $e\in V$ is called {\it Archimedean} if for each
$v\in V$, we have $v \leq 0$ whenever $\lambda v \leq e$ for all $\lambda \geq 0$. We call $V$ an {\it Archimedean order-unit space}
if it is equipped with an Archimedean order-unit.
An Archimedean order-unit $e\in V$ induces a norm $\|\cdot\|_e$ on $V$, called an {\it order-unit norm}, and is defined by
 \[ \|x\|_e = \inf\{\lambda>0: -\lambda e \leq x \leq \lambda e\} \qquad (x\in V)\]
which satisfies $-\|x\|_e e\leq x \leq \|x\|_e e $.
We denote by $(V,e)$ an Archimedean order-unit space $V$ equipped with an Archimedean order
unit $e$ and  the order-unit norm $\|\cdot\|_e$, where the subscript
$e$ will be omitted if it is understood. We call $(V,e)$ a {\it complete Archimedean order-unit space} if the order-unit norm $\|\cdot\|_e$ is complete.

We note that the cone $V_+$ in $(V,e)$ is closed in the order-unit norm and the  Archimedean order-unit $e$
 belongs to the interior
${\rm int}\, V_+$ of  $V_+$ \cite[Theorem 2.2.5]{ellis}. In fact, each interior point $u\in {\rm int}\,V_+$ is an order-unit
(cf.\,Lemma \ref{e})
and the corresponding order-unit norm $\|\cdot\|_u$ is equivalent to $\|\cdot \|_e$.

A linear map $T: V \rightarrow W$ between two order-unit spaces $(V,e)$ and $(W,u)$ is called {\it positive}
if $T(V_+) \subset W_+$. It is called a {\it positive linear functional} on $V$ if $(W,u)=(\mathbb{R},1)$.
The dual $V^*$ of $(V,e)$ is partially ordered by the dual cone $V_+^*$, which
consists of continuous positive  linear functionals on $V$ and is precisely the polar set
\begin{equation}\label{sd}
- V_+^0 := - \{f\in V^*: |f(x)|\leq 1, \forall x\in V_+\}
 =\{ f\in V^*: f(x) \geq 0, \forall x \in V_+\}
 \end{equation}
(cf.\,\cite[p.30]{ellis}).

\begin{lemma}\label{cts} Let  $T: V \rightarrow W$ be a positive linear map between two Archimedean order-unit spaces $(V,e)$ and $(W,u)$.
Then $T$ is continuous and $\|T\|= \|T(e)\|_u$.
\end{lemma}
\begin{proof} We need to show $\sup \{\|Tx\|_u: -e \leq x \leq e\} <\infty$, where $\{x\in V: -e\leq x \leq e\}$ is
  the closed unit ball of $(V,e)$. Indeed, given $-e\leq x \leq e$, we have
$$-\|T(e)\|_u u \leq -T(e) \leq T(x) \leq T(e)\leq \|T(e)\|_u u$$ by positivity. This implies $\|T(x)\|_u \leq \|T(e)\|_u$
and hence $T$ is continuous with $\|T\| \leq \|T(e)\|_u$. Since $\|T\|=\sup \{\|Tx\|_u: -e \leq x \leq e\}  \geq \|T(e)\|_u$,
we have $\|T\|= \|T(e)\|_u$.
\end{proof}

\begin{proposition} Let $T: (V,e) \rightarrow (V,e)$ be a linear isomorphism such that $T(V_+) = V_+$. Then $T$ is an isometry
if, and only if, $T(e)=e$.
\end{proposition}
\begin{proof}
Let $T$ be an isometry. By Lemma \ref{cts}, we have $\|T(e)\|=\|T\|=1=\|T^{-1}\|=\|T^{-1}(e)\|$ which implies
$-e\leq T(e) \leq e$ and $-e \leq T^{-1}(e) \leq e$.  Therefore $T(e)=e$ by positivity.

Conversely, Lemma \ref{cts} implies $$\|T\|=\|T(e)\|=\|e\|=1=\|T^{-1}(e)\|=\|T^{-1}\|.$$ Hence $T$ is an isometry.
\end{proof}

We now introduce the concept of an infinite dimensional symmetric cone (cf.\,\cite[p.\,105]{chu})
which is a natural generalisation of the finite dimensional one.

\begin{definition}
Let $V$ be a real Hilbert space with inner product $\langle \cdot,\cdot\rangle$. An open cone
$\Omega$ in $V$ is called {\it symmetric} if it satisfies the following two conditions:
\begin{enumerate}
\item [(i)] (self-duality) $\Omega = \{v\in V: \langle v,x\rangle >0, \forall x\in \overline\Omega \backslash\{0\}\}$;
\item[(ii)] (homogeneity) given $x,y\in \Omega$, there is a continuous linear isomorphism $g: V\rightarrow V$
such that $g(x)=y$.
\end{enumerate}
\end{definition}

Since a cone $\Omega$ in a Hilbert space $V$ is convex, its
weak and norm closures in $V$ coincide and is denoted by $\overline \Omega$.  If $\Omega$ is open,
we have $\Omega ={\rm int}\,\overline \Omega$ from Lemma \ref{2.1}.
By self-duality, the closure $\overline \Omega$ of a symmetric cone $\Omega$ is proper and in view of (\ref{sd}),
also {\it `self-dual'} in the sense of
Connes \cite{con}, namely,
$$ \overline \Omega = \{v\in V:\langle v,x\rangle \geq 0, \forall x\in \overline \Omega\}.$$
Indeed, given $v \in V$ with $\langle v,x \rangle \geq 0$ for all $x \in \overline \Omega$, we have, by picking some $e \in \Omega$,
$$ \langle v + \frac{1}{n} e, x\rangle = \langle v,x\rangle + \frac{1}{n}\langle e,x \rangle > 0 \qquad (n=1,2, \ldots)$$
for $x \in \overline \Omega \backslash \{0\}$. Hence $v + \frac{1}{n} e \in \Omega$ for $n=1,2,\ldots$ and $v \in  \overline \Omega$.

For finite dimensional Euclidean spaces, the preceding definition is the same as the usual one for a symmetric cone \cite{fk}.
In finite dimensions, Koecher and Vinberg's celebrated result states that the interior of the cone $\{x^2: x\in \mathcal{A}\}$ in a
formally real Jordan algebra $\mathcal{A}$ is a symmetric cone and conversely, every symmetric cone is of this form.
We will extend this result to the infinite dimensional setting in the next section. We first discuss how a symmetric cone
relates to the underlying Hilbert space structure.

\begin{lemma}\label{e} Let $V$ be a real vector space, equipped with a norm $\|\cdot\|$ and partially ordered by the closure $\overline \Omega$
of an open cone $\Omega$. Then each point $e\in \Omega$ is an order-unit. If moreover, $e$ is Archimedean,
 then the order-unit norm $\|\cdot\|_e$
satisfies $\|\cdot\|_e \leq c\|\cdot\|$ for some $c>0$.
\end{lemma}

\begin{proof} Let $e\in \Omega$. Since $\Omega$ is open, there exists $r>0$ such that the open ball $e-B(0,r)$
is contained in $\Omega$, where $B(0,r) = \{x\in V: \|x\| < r\}$. Let $v\in V\backslash \{0\}$.
Then we have $\pm (r/2\|v\|)v \in B(0,r)$
which implies $ e \mp (r/2\|v\|)v \in e -B(0,r) \subset \Omega$, that is
$$- \frac{2\|v\|}{r} e \leq v \leq \frac{2\|v\|}{r} e.$$ This proves that $e$ is an order-unit.
If $e$ is Archimedean, it also implies that the order-unit norm $\|\cdot\|_e$ satisfies
$\|v\|_e \leq (2/r)\|v\|$ for all $v \in V$.
\end{proof}

From now on, the inner product norm of a Hilbert space $V$ will always
be denoted by $\|\cdot\|= \langle \cdot,\cdot\rangle
^{1/2}$. The following result reveals that a real Hilbert space equipped with a symmetric cone is a
complete Archimedean order-unit space.

\begin{lemma} Let $V$ be a real Hilbert space partially ordered by the closure $\overline \Omega$
of a symmetric cone $\Omega$ and let $e\in \Omega$.
Then $e$ is an Archimedean order-unit and the order-unit norm $\|\cdot\|_e$ is equivalent to the inner product norm $\|\cdot\|$ of $V$.
\end{lemma}
\begin{proof} By lemma \ref{e}, $e$ is an order-unit. To see that $e$ is Archimedean, let $\lambda v \leq e$ for all $\lambda
\geq 0$. By self-duality of $\Omega$, we have $\langle e-\lambda v, x\rangle \geq 0$ for all $x\in \overline\Omega$, which gives
$\langle e,x\rangle \geq \lambda\langle  v, x\rangle$ for all $\lambda \geq 0$. If $x\in \overline\Omega\backslash\{0\}$, then
$\langle e, x \rangle >0$. It follows that $\langle  v, x\rangle \leq 0$ for all $x\in \overline
\Omega$ and hence $-v \in \overline\Omega$ by self-duality.

For the second assertion, Lemma \ref{e} already implies  $\|\cdot\|_e \leq c\|\cdot\|$ for some $c>0$.

To complete the proof, we show that $\|v\|^2 \leq \langle e,e\rangle \|v\|_e^2$ for all $v\in V$. Indeed, we have
$\|v\|_e e\pm v \in \overline\Omega$ and by self-duality, $\langle \|v\|_e e + v, \|v\|_e e - v\rangle \geq 0$. Expanding the inner product
gives $\|v\|_e^2 \langle e,e\rangle \geq \|v\|^2$.
\end{proof}

Let $V$ be a real Hilbert space partially ordered by the closure of a symmetric cone $\Omega$ and let $e\in \Omega$. Then the previous
lemma implies that a linear map $T : V \rightarrow V$ is continuous with respect to the Hilbert space norm of $V$ if, and only if,
it is continuous with respect to the order-unit norm $\|\cdot\|_e$. In the sequel, we will denote by $\|T\|$ and $\|T\|_e$
the operator norm of $T: V \rightarrow V$ with respect to the Hilbert space norm and order-unit norm of $V$, respectively.

We refer to \cite{bou,up} for definitions and properties of infinite dimensional Banach Lie groups and Lie algebras, which are
analytic manifolds.
Let  $L(V)$ be the Banach algebra of bounded linear operators on $V$, equipped with the Hilbert space operator norm $\|\cdot\|$
and the usual involution $*$.
Then $L(V)$ is a real Banach Lie algebra in the commutator product $[S,T] = ST - TS$ for $S,T\in L(V)$.
We denote by $GL(V)$ the open subgroup of invertible elements in $L(V)$, which is a real Banach Lie group with Lie
 algebra $L(V)$.
Given $S,T \in GL(V)$, we have $\|S^{-1}-T^{-1}\| \leq \|S^{-1}\|\|S-T\|\|T^{-1}\|$. Hence
 the orthogonal group $$O(V)=\{T\in GL(V): \|T\|=\|T^{-1}\|=1\}$$ of $V$,
 consisting of linear isometries of $V$ and equipped with
the norm topology, is a closed subgroup of $GL(V)$ and a real Banach Lie group.

In contrast to the finite dimensional case, a closed subgroup $H$ of an infinite dimensional real Banach Lie group $G$
need not be a Lie group in the relative topology \cite{hk}. Nevertheless,  it can still be topologised (with a finer topology
$\mathcal{T}$)
 to form a Banach Lie group, by \cite[7.8]{up}. In fact, if $\frak g$ is the Lie algebra of $G$, then the Lie algebra of $H$ is given by
$$\frak h = \{X\in \frak g: \exp tX \in H, \forall t\in \mathbb{R}\}$$ and
the inclusion map $(H,\mathcal{T}) \hookrightarrow G$ is analytic.

Henceforth, let $\Omega$ be a symmetric cone in a real Hilbert space $V$.
The open cone $\Omega$ is a real Hilbert manifold modelled on
$V$, where the tangent space $T_\omega \Omega$ at $\omega \in \Omega$ is identified with
$V$. The positive linear maps in $GL(V)$ with positive inverse,
 with respect to the cone
$\Omega$, form a subgroup of $GL(V)$, which will be denoted by
$$ G(\Omega) = \{g\in GL(V): g(\Omega) = \Omega\}.$$
An element $g\in GL(V)$ belongs to
$G(\Omega)$ if and only if $g(\overline\Omega) = \overline\Omega$. Hence $G(\Omega)$ is a closed subgroup of
 $GL(V)$ and can be topologised in a finer topology to
 a real Banach Lie group with Lie algebra
$$\frak g(\Omega) =\{ X\in L(V): \exp t X \in G(\Omega),  \forall t \in \mathbb{R}\}.$$

For each $\omega \in \Omega$, the map $g\in GL(V) \mapsto g(\omega) \in V$ is analytic and its
derivative at the identity of $GL(V)$  is the linear map $X\in L(V) \mapsto X(\omega) \in V$.
Since the inclusion map $G(\Omega)\hookrightarrow GL(V)$ is analytic, the orbital map
$g\in G(\Omega) \mapsto g(\omega) \in \Omega$ is analytic.

By a {\it Lie subgroup} of $G(\Omega)$, we mean a subgroup and submanifold of $G(\Omega)$. For instance,
the connected component $G_0$ of the identity in $G(\Omega)$ is a Lie subgroup of $G(\Omega)$.
 \cite[Chap.\,III,\S1.3]{bou}.
The subgroup $K=G(\Omega) \cap O(V)$ of $G(\Omega)$ is closed in the norm and
any finer topologies of $G(\Omega)$, and also a real Banach Lie group.

Given $g\in G(\Omega)$, we note that its adjoint $g^*$ in $L(V)$ also belongs to $G(\Omega)$. Indeed, for $v\in \Omega$ and $x\in \overline
\Omega\backslash \{0\}$, we have $\langle g^*(v), x\rangle = \langle v, g(x)\rangle >0$ and hence $g^*(v) \in \Omega$.

\begin{lemma}\label{alpha} Let $V$ be a real Hilbert space partially ordered by the closure of
a symmetric cone $\Omega$ and let $e\in \Omega$
with $\alpha \|\cdot\| \leq \|\cdot\|_e \leq \beta \|\cdot\|$ for some $ \beta >\alpha >0$. Then we have
$\frac{\alpha}{\beta} e \leq g(e) \leq \frac{\beta}{\alpha} e$  for all $g \in G(\Omega) \cap O(V)$.
\end{lemma}
\begin{proof} Let $g \in G(\Omega) \cap O(V)$. Then we have
$$\frac{\alpha}{\beta} \|g\|\leq \|g\|_e \leq \frac{\beta}{\alpha}\|g\|$$
where $\|g\|=1$.
By Lemma \ref{cts}, $\|g(e)\|_e = \|g\|_e$ since $g(\Omega) \subset \Omega$, and the same for $g^{-1}$. It follows that
$$\frac{\alpha}{\beta} e\leq g(e) \leq \frac{\beta}{\alpha}e.$$
\end{proof}

\begin{theorem}\label{fix} Let $V$ be a real Hilbert space partially ordered by the closure of
a symmetric cone $\Omega$. Then there exists
$\omega\in \Omega$ such that $g(\omega) = \omega$ for all $g \in G(\Omega) \cap O(V)$.
\end{theorem}

\begin{proof} Fix an Archimedean order-unit  $e\in \Omega$ and let
$\alpha \|\cdot\| \leq \|\cdot\|_e \leq \beta \|\cdot\|$ for some $ \beta >\alpha >0$. By Lemma \ref{alpha}, the orbit
$$S = \{g(e): g \in G(\Omega) \cap O(V)\}$$
is contained in the order interval
$$[[ ({\alpha}/{\beta}) e,\, ({\beta}/{\alpha})e]] :=
\{x\in V: ({\alpha}/{\beta}) e \leq x \leq ({\beta}/{\alpha})e\}$$
which is convex, bounded in the order-unit norm $\|\cdot\|_e$ and hence bounded in the Hilbert space norm of $V$.
It is clearly closed in the order-unit norm, and hence in the Hilbert space norm of $V$. It follows that the
order interval $\left[\left[ ({\alpha}/{\beta}) e,\, (\beta/\alpha)e\right]\right]$ is weakly compact in $V$.
Let $Q = \overline{\rm co}\, S$ be the closed convex hull of the orbit $S$. Then $Q \subset [[ ({\alpha}/{\beta}) e,\, ({\beta}/{\alpha})e]]$ and is weakly compact.

Since $G(\Omega)\cap O(V)$ is a group and each $g \in O(V)$ is continuous on $V$, we see that $g(Q) \subset Q$
for all $g\in G(\Omega)\cap O(V)$. Further, $G(\Omega) \cap O(V)$ is equicontinuous
on $Q$ in the weak topology (as defined in \cite[V.10.7]{d}), that is, given
any weak neighbourhood $N$ of $0$, there exists a weak neighbourhood $U$ of $0$ such that for $x,y\in Q$ with $x-y\in U$,
 we have $g(x-y) \in N$ for all $g\in G(\Omega) \cap O(V)$.

To prove equicontinuity, we first observe that for any weak neighbourhood of $0$ of the form $N_0=\{v\in V: |\langle v,z\rangle|< \varepsilon\}$
for some $z\in V\backslash\{0\}$, one can find $z_1,z_2\in \overline\Omega \backslash\{0\}$ such that $N_0$ contains the following weak
neighbourhood of $0$:
$$\{v\in V: |\langle v,z_j\rangle|< \varepsilon/2, j=1,2\}.$$
Indeed, write $z=z_1 - z_2$ with $z_1, z_2 \in \overline\Omega \backslash\{0\}$, then we have
$$|\langle v,z\rangle|= |\langle v,z_1 -z_2\rangle| \leq |\langle v,z_1\rangle|
+ |\langle v,z_2\rangle|< \varepsilon.$$
Hence, given any weak neighbourhood $N$ of $0$, there are positive elements $z_1, \ldots, z_k$
in $\overline\Omega \backslash\{0\}$ such
that
$$N \supset \{v\in V: |\langle v,z_j\rangle|< \varepsilon, j=1, \ldots k\}$$
for some $\varepsilon >0$. We can find $c>1$ such that $z_j \leq c e$ for $j=1, \ldots,k$.
Now, for each $g \in G(\Omega) \cap O(V)$ and $j=1, \ldots,k$, we have
$$0 \leq g(z_j)\leq cg(e) \leq \frac{c\beta}{\alpha} e$$
which gives
\begin{equation}\label{8}
0\leq \langle x, g(z_j)\rangle \leq \left\langle x, (c\beta/\alpha) e\right\rangle
\end{equation}
for all $x \in \overline\Omega$. Pick $\gamma \in (0,1)$ such that
$\displaystyle\gamma <  \frac{\varepsilon\alpha^2}{c(\beta^2-\alpha^2)\langle e,e\rangle}$. Then
$$U=\{v\in V: |\langle  \gamma(v +(\beta^2-\alpha^2)/(\alpha\beta)\, e),\, (c\beta/\alpha)\,e\rangle | <\varepsilon\}.$$
is a weak neighbourhood of $0$ in $V$.
Since $(\alpha/\beta)e \leq g(e) \leq (\beta/\alpha)e$  for all $g \in G(\Omega)\cap O(V)$, by Lemma
\ref{alpha}, we have
$$(\alpha/\beta)\langle e,e\rangle \leq \langle g(e),e\rangle \leq (\beta/\alpha)\langle e,e\rangle
$$ and hence
$$0< \frac{\gamma(\beta^2-\alpha^2)}{\beta^2}\langle e,e\rangle \leq
\frac{\gamma(\beta^2-\alpha^2)}{\alpha \beta}\langle g(e),e\rangle
\leq \frac{c\gamma(\beta^2-\alpha^2)}{\alpha^2}\langle e,e\rangle <\varepsilon$$
 for all $g \in G(\Omega)\cap O(V)$.

Let $x,y\in Q \subset \left[\left[ ({\alpha}/{\beta}) e,\, (\beta/\alpha)e\right]\right]$.
We have
$$ -\left(\frac{\beta}{\alpha} - \frac{\alpha}{\beta}\right)e \leq x-y \leq \left(\frac{\beta}{\alpha} - \frac{\alpha}{\beta}\right)e = \frac{\beta^2-\alpha^2}{\alpha\beta}\,e$$
and in particular,  $0\leq  (x-y)+(\beta^2-\alpha^2)/(\alpha\beta)\,e \leq (x-y)/\gamma
+(\beta^2-\alpha^2)/(\alpha\beta)\, e $ where $\gamma \in (0,1)$. Hence $x-y \in \gamma U$ implies
\begin{eqnarray*}
&&0\leq \langle g(\gamma((x-y)/\gamma + (\beta^2-\alpha^2)/(\alpha\beta)\, e),\, z_j\rangle \\
&=& \langle \gamma\{(x-y)/\gamma + (\beta^2-\alpha^2)/(\alpha\beta)\, e\},\, g^*(z_j)\rangle \\
&\leq& \langle \gamma\{(x-y)/\gamma +(\beta^2-\alpha^2)/(\alpha\beta)\, e\},\, (c\beta/\alpha)e
\rangle  <\varepsilon
\end{eqnarray*}
by (\ref{8}), which gives
$$  - \frac{\gamma (\beta^2-\alpha^2)}{\alpha\beta}\langle g(e),z_j\rangle \leq \langle  g(x-y), z_j\rangle \leq
\varepsilon - \frac{\gamma (\beta^2-\alpha^2)}{\alpha\beta}\langle g(e),z_j\rangle < \varepsilon$$
where (\ref{8}) implies
$$\frac{\gamma (\beta^2-\alpha^2)}{\alpha\beta}\langle g(e),z_j\rangle \leq \frac{\gamma (\beta^2-\alpha^2)}{\alpha\beta}\langle e, \frac{c\beta}{\alpha}e\rangle <\varepsilon$$
and hence
$$|\langle g(x-y), z_j\rangle| < \varepsilon $$
for all $g \in G(\Omega)\cap O(V)$ and $j=1, \ldots, k$. This proves equicontinuity of $G(\Omega) \cap O(V)$.

Hence, by Kakutani's fixed-point theorem \cite[V.10.8]{d}, the group $G(\Omega) \cap O(V)$
has a common fixed point, say, $\omega \in Q \subset [[ ({\alpha}/{\beta}) e,\, ({\beta}/{\alpha})e]]
\subset \Omega$. The proof is complete.
\end{proof}

We note that a bounded linear operator $T$ on a {\it complex} Hilbert space $H$ is hermitian if, and only if, $\|\exp itT\| = 1$ for
all $t \in \mathbb{R}$ (cf.\,\cite[p.\,46]{bon}). If  $\|\exp itT\|\leq M$ for some $M>0$ and
all $t \in \mathbb{R}$, then $\exp itT$  has spectral radius
$$\rho(\exp itT)=\lim_{n\rightarrow \infty}\|(\exp itT)^n\|^{1/n} = \lim_{n\rightarrow \infty} \|\exp intT\|^{1/n} \leq \lim_{n\rightarrow
\infty} M^{1/n} = 1$$ and if $iT$ is hermitian as well, then we have $\|\exp itT\|= \rho(\exp itT) \leq 1$ for all
$t\in \mathbb{R}$,
which implies that $T$ is hermitian and hence $T=0$.
 Given $X \in L(V)$, by considering its complexification $X_{\mathbb{C}}: V_{\mathbb{C}} \rightarrow V_{\mathbb{C}}$
 on the complex Hilbert space $V_{\mathbb{C}}$ and noting that $\exp tX$ is an isometry on $V$ if, and only if,
 $\exp tX_{\mathbb{C}}$ is a unitary operator on $V_{\mathbb{C}}$, we see that
$\|\exp tX\|=1$ for all $t\in \mathbb{R}$ if, and only if, $X$ is skew-symmetric, that is, $X^* = -X$. On the other hand,
if $X^*=X$ and if $\|\exp tX\|\leq M$
for some $M>0$ and for all $t \in \mathbb{R}$, then we must have $X=0$.

The above observation implies that the Lie algebra of the Lie group $K= G(\Omega) \cap O(V)$ is given by
\begin{eqnarray*}
\frak k &=& \{X\in \frak g(\Omega): \exp tX \in G(\Omega) \cap O(V), \forall t\in \mathbb{R}\}\\
&=& \{X\in \frak g (\Omega): X^* =-X\} \subset L(V).
\end{eqnarray*}

For $X\in \frak g(\Omega)$, we have $\exp tX \in G(\Omega)$ and $\exp tX^* = (\exp t X)^* \in G(\Omega)$ for all $t \in \mathbb{R}$.
Hence $X^* \in
\frak g(\Omega)$. 
Let
$$ \frak p = \{ X \in \frak g(\Omega): X^* =X\}.$$
Then we have the direct sum decomposition
$\frak g (\Omega)= \frak k \oplus \frak p$ 
with Lie brackets
$$[ \frak k, \frak p] \subset \frak p , \qquad [\frak p, \frak p ] \subset \frak k.$$
This implies that the inclusion map $\iota: K \hookrightarrow G(\Omega)$
is an immersion, as  its differential
at the identity, $d\iota :\frak k \rightarrow \frak g(\Omega)$, has an image with a direct sum complement $\frak p$
in $\frak g(\Omega)$
(cf.\,\cite[Definition 2.1.16]{chu}). 

By Theorem \ref{fix}, there is an order-unit $\omega \in \Omega$ which is a fixed-point of
the group $K= G(\Omega) \cap O(V)$. Given $X\in \frak k$, we have $\exp tX \in K $
and
$$\omega = \exp t X (\omega) = \omega + t X(\omega) + \frac{t^2}{2!} X(\omega) + \cdots \quad
\mbox{for all}~ t\in \mathbb{R}$$
which implies $X(\omega) = \displaystyle \left.\frac{d}{dt} \right|_{t=0}\exp t X(\omega)=0$. In fact, the converse
also holds.
\begin{lemma}\label{9} Let $\omega \in \Omega$ be a fixed point of $K= G(\Omega)\cap O(V)$ with Lie algebra $\frak k$.
Then we have
$$\frak k = \{ X\in \frak g(\Omega): X(\omega) = 0\}.$$
\end{lemma}
\begin{proof}
Let $X\in \frak g(\Omega)$ and $X(\omega)=0$. We have the decomposition $X= X_{\frak k} + X_{\frak p} \in \frak k \oplus \frak
p$. Since $X_{\frak k} (\omega) =0$ as noted above, we have $X_{\frak p}(\omega)=0$ and hence
$\exp t X_{\frak p} (\omega) = \omega$ for all $t \in \mathbb{R}$. This implies $\|\exp tX_{\frak p}\|_{\omega} =
\|\omega\|_{\omega}= 1$
and there exists $M>0$ such that $\|\exp t X_{\frak p}\| \leq M$ for all $t\in \mathbb{R}$ since the order-unit norm
$\|\cdot\|_{\omega}$ is equivalent to the Hilbert space norm $\|\cdot\|$. On the other hand, $X_{\frak p} \in \frak p$ implies
$X_{\frak p}^* =X_{\frak p}$. Therefore we must have $X_{\frak p} =0$ and $X=X_{\frak k} \in \frak k$.
\end{proof}

\begin{remark}\label{2.9}
The above lemma implies that the Lie algebra
 $$\frak k_{\omega} =\{X\in \frak g(\Omega) : \exp tX (\omega) = \omega,\, \forall t\in \mathbb{R}\}$$
 of   the isotropy subgroup $$K_{\omega}
= \{ g\in G(\Omega): g(\omega)=\omega\} \supset K$$ coincides with the Lie algebra $\frak k$ of
$K$, where $\frak g(\Omega)= \frak k \oplus \frak p = \frak k_\omega \oplus p$. Hence, analogous to the case of $K$, the inclusion map  $K_{\omega} \hookrightarrow G(\Omega)$ is an immersion.
\end{remark}

The differential of the orbital map $\rho_\omega: g\in G(\Omega) \mapsto g(\omega) \in \Omega$ at the identity of $G(\Omega)$ is the
evaluation map $X\in \frak g(\Omega) \mapsto X(\omega) \in V$.
By homogeneity of $\Omega$, we can identify the Lie algebra $\frak g(\Omega)= \frak k_\omega \oplus \frak p$
 with the Lie algebra $\frak{aut}\, \Omega$ of analytic  vector fields on $\Omega$
 that generate one-parameter subgroups of $G(\Omega)$ (cf.\,\cite[p.\,110]{w}).
 This implies that the evaluation map $X \in \frak g(\Omega) \mapsto X(\omega) \in T_\omega \Omega =V $ is surjective.
 It follows that the orbital map $\rho_\omega$ 
is a submersion since $\frak k_\omega$ is the kernel of the evaluation map, which is
complemented in $\frak g(\Omega)$. In particular, $\rho_\omega$ is an open map and
the identity component $G_0$ also acts transitively on $\Omega$ since $\Omega$
is connected and a disjoint union of open $G_0$-orbits.

\section{JH-algebras}

In this final section, we extend to infinite dimension the
celebrated result of
Koecher \cite{k} and Vinberg \cite{v} on the one-one correspondence between finite dimensional symmetric cones and formally real
Jordan algebras. This correspondence is furnished by the assertion that given a finite dimensional formally real Jordan algebra $\mathcal{A}$, which is a Hilbert space in the trace norm, the interior of the cone $\{x^2: x\in \mathcal{A}\}$ is a symmetric cone and, every symmetric cone in an Euclidean space is of this form.
We show in Theorem \ref{symm} below that symmetric cones of all dimensions are in one-one correspondence with the unital JH-algebras.

Let $\mathcal{H}$ be a JH-algebra with identity $\mathbf{1}$ and let
$$\mathcal{H}_+ = \{x^2: x\in \mathcal{H}\}.$$
It has been shown in \cite[Lemma 2.3.17]{chu} that $\mathcal{H}_+$ is a closed cone and its interior
\begin{equation}\label{0}
{\rm int}\,\mathcal{H}_+ = \{ y\in \mathcal{H}_+ : \langle y, a\rangle >0, \forall a\in \mathcal{H}_+\backslash\{0\}\}
\end{equation}
is a symmetric cone consisting of invertible elements in $\mathcal{H}_+$. For each $x^2\in \mathcal{H}_+$, the left multiplication $L_{x^2} : \mathcal{H} \rightarrow \mathcal{H}$ is a positive self-adjoint operator on the Hilbert space
 $\mathcal{H}$ \cite[p.\,108]{chu}.
  Further, ${\rm int}\,\mathcal{H}_+$ is a Hilbert manifold
and a Riemannian symmetric space, which can be infinite dimensional, with Riemannian metric
$$g_{\omega}(u,v) = \langle Q_{\omega^{-1}} ( u), \, v\rangle \qquad (\omega \in {\rm int}\,\mathcal{H}_+,
\,u, v \in \mathcal{H})$$
where the quadratic map $Q_{x^{2}}= 2 L_{x^2}^2 - L_{x^4}= Q_x^2$ is a positive  operator on $\mathcal{H}$,
for $x^2\in\mathcal{H}_+$.

In what follows, we retain the notation in the previous section.

\begin{theorem}\label{symm} Let $\Omega$ be an open cone  in a real Hilbert space. Then $\Omega$ is a symmetric cone
if, and only if, it is of the form
\begin{equation}\label{1.1} \Omega = {\rm int}\,\{a^2: a\in \mathcal{H}\}
\end{equation}
for a unique unital JH-algebra $\mathcal{H}$.
\end{theorem}

\begin{proof} Given a unital JH-algebra $\mathcal{H}$, the cone
\[ \Omega = {\rm int}\,\{a^2: a\in \mathcal{H}\}\]
is a symmetric cone by \cite[Lemma 2.3.17]{chu}, as noted above.

Conversely, let $\Omega$ be a symmetric cone in a real Hilbert space $V$. We show that it is of the form in
(\ref{1.1}) for a unique JH-algebra $\mathcal{H}$. In fact, we show that $\mathcal{H}$ is the Hilbert space
$V$ itself, equipped with a suitable Jordan product.
Our arguments are a natural extension, albeit with some infinite dimensional adaptation,
of Satake's proof in \cite[Theorem I.8.5]{satake}
for the finite dimensional case (see also \cite{fk}).

We first note that the JH-algebra $\mathcal{H}$ in (\ref{1.1}) must be unique since
$(\mathcal{H}, \Omega)$ is an order-unit space and hence $\mathcal{H}= \overline\Omega - \overline\Omega$.

As before, let $K = G(\Omega) \cap O(V)$. By Theorem \ref{fix},  $K$ is contained in the isotropy subgroup $K_\omega$ of $G(\Omega)$ at some $\omega \in \Omega$. The Lie algebra $\frak g(\Omega)$ of $G(\Omega)$ has a direct sum decomposition
$$\frak g(\Omega) = \frak k_\omega \oplus \frak p$$
where $\frak k_\omega$ is the Lie algebra of $K_\omega$.

As noted at the end of the previous section, homogeneity of $\Omega$ implies that  the evaluation map
$$X\in \frak g(\Omega) \mapsto X(\omega)\in V$$
is surjective. It follows that
 the map $\Phi: X\in \frak p \mapsto X(\omega)\in V$ is a linear isomorphism since it
has
kernel
$$\{X\in \frak g: X(\omega)=0\}=\frak k_\omega$$
by Lemma \ref{9} and Remark \ref{2.9}, where  $\frak k_\omega \cap \frak p =\{0\}$.
Denote the inverse of $\Phi$ by
$L: V \rightarrow \frak p$ so that
$$L(x)\omega = x \qquad (x \in V).$$
On $V$, we defined a product
$$xy := L(x)y = L(x)L(y) \omega \qquad (x,y \in V).$$
We show that $V$ is a Jordan algebra with this product and moreover, it is a JH-algebra with identity $\omega$
such that
$$\Omega = {\rm int}\,\{x^2: x\in V\}$$
which would complete the proof.

To see that $V$ is a Jordan algebra in the above product, let $x,y\in V$. Then we have
$$xy = L(x)y =L(x)L(y) \omega \quad {\rm and} \quad xy-yx = [L(x), L(y)]\omega =0$$
where $L(x), L(y) \in \frak p$ implies $[L(x), L(y)] \in \frak k_\omega$. To prove the Jordan identity, we observe
\[x^2(yx) = x^2 (L(y)L(x)\omega) = L(x^2)L(L(y)L(x)\omega)\omega = L(x^2)L(y)L(x)\omega\]
and
\[ x(x^2y) = L(x) L(L(x^2)L(y)\omega)\omega = L(L(x^2)L(y)\omega)L(x)\omega =L(x^2)L(y)L(x)\omega\]
which equals $x^2(yx)$. This proves that $V$ is a Jordan algebra.

Evidently $\omega$ is the identity in $V$ as $x\omega = L(x)\omega =x$ for all $x\in V$. To see that the inner product
$\langle \cdot, \cdot \rangle$ is associative, we first note that $L(x)^* = L(x)$ since $L(x) \in \frak p$.
Given $x,y,z \in V$, we have
\[\langle xy, z\rangle = \langle L(x)L(y)\omega, z\rangle =\langle L(y)\omega, L(x)z\rangle = \langle y, xz\rangle\]
which shows that $V$ is a JH-algebra.

Finally, we show that $\Omega$ is identical with the symmetric cone
$$ C := {\rm int}\,\{x^2: x\in V\}.$$
Let $x^2 \in C$. We have noted previously that the left multiplication $L_{a^2}$ is a positive self-adjoint operator
on the Hilbert space $V$, for each $a\in V$. Since $x^2$ is an interior point of the cone, it has been shown in \cite[p.\,109]{chu} that $x^2 = a^2 + \beta \omega$
 for some $a\in V$ and $\beta >0$. This implies that $L_{x^2} = L_{a^2} + \beta L_\omega$, where $L_\omega$ is the
 identity operator on $V$. Hence $L_{x^2}$ is
  a positive invertible operator in $L(V)$ and by spectral theory, there exists $T\in L(V)$
such that $L_{x^2} = \exp T$. It follows that
$$x^2  = L_{x^2}(\omega) = \exp T (\omega) = \exp X (\omega) $$
for some $X\in \frak p$. Since $\exp X \in G(\Omega)$,
we have $\exp X(\omega) \in \Omega$, that is, $x^2 \in \Omega$. We have shown that
$C \subset \Omega$ and they must be equal by self-duality.

\end{proof}

In finite dimensions, the above theorem is exactly the aforementioned result of Koecher and Vinberg since
the finite dimensional unital JH-algebras coincide with the formally real Jordan algebras. The following corollary
of the theorem is immediate.

\begin{corollary} A symmetric cone in a Hilbert space $V$ carries the structure of a Riemannian symmetric space.
\end{corollary}

We conclude with two examples of infinite dimensional unital JH-algebras.

\begin{example} Given any real Hilbert space $H$ with inner product $\langle\cdot,\cdot\rangle$, the Hilbert space direct sum $H\oplus\mathbb{R}$, called a {\it spin factor}, is a JH-algebra with identity $0\oplus 1$ in the following Jordan product
$$(a\oplus \alpha)(b\oplus \beta) := (\beta a +\alpha b) \oplus
(\langle a, b\rangle +\alpha \beta).$$ In particular, for $a,b\in H$, we have
$(a\oplus 0)(b\oplus 0) = \langle a, b\rangle (0\oplus 1)$, which is a scalar multiple of
the identity.
\end{example}

Given a spin factor $V$, we denote by $V_0$ the orthogonal complement of the identity in $V$. In the above example, $V_0$
is just the Hilbert space $H$. For each $v\in V_0$, we observe that $v^2$ is a scalar multiple of the identity.

\begin{example} Let $Z= V \oplus W$  be the Hilbert space direct sum of two spin factors $V$ and $W$, equipped with
coordinatewise Jordan product. Then it can be verified readily that $Z$
is a JH-algebra with identity $(e_V, e_W)$, where $e_V$ is the identity of $V$ and $e_W$ that of $W$.
Further, $Z$ is not a spin factor. Indeed, the orthogonal complement $Z_0$
 of the identity contains $V_0 \oplus W_0$ and for $z= (v,w) \in V_0\oplus W_0$, we have $z^2= (v^2, w^2)$
which need not be a scalar multiple of  $(e_V, e_W)$.
\end{example}

\end{document}